\documentclass[10pt,a4paper]{article}
\usepackage{amssymb}
\usepackage{amsfonts,amsmath,latexsym}
\usepackage{indentfirst}
\usepackage{CJK}
\usepackage{graphicx}
\usepackage{amsthm}
\usepackage{float}
\usepackage{slashbox}

\newtheorem{thm}{Theorem}[section]
\newtheorem{lem}{Lemma}[section]
\newtheorem{remark}{Remark}[section]
\newtheorem{coro}{Corollary}[section]

\numberwithin{equation}{section}
\numberwithin{figure}{subsection}

\title{Uniformly stable rectangular elements for fourth order elliptic singular perturbation problems
\thanks{This work was supported in part by the  Natural Science Foundation of China (10771150),
the National Basic Research Program of China (2005CB321701),
and the Program for New Century Excellent Talents in University (NCET-07-0584) }}

\author{Li Wang \thanks{Email: liwscu@yahoo.cn }\hspace{15mm}Xiaoping Xie
\thanks{Corresponding author. Email:xpxie@scu.edu.cn}\\
 { \small School of Mathematics, Sichuan University, Chengdu 610064,  China}
}

\date{}
\begin{document}
\maketitle \pagestyle{plain}

\begin{center}

\begin{abstract}
This paper analyzes rectangular finite element methods for fourth order elliptic singular
perturbation problems.  We show that the non-$C^0$ rectangular Morley element is uniformly
convergent in the energy norm with respect to the perturbation parameter.  We also propose a
$C^0$ extended high order rectangular Morley element and prove the uniform convergence.
Finally, we do some numerical experiments to confirm the theoretical results.
par
\vspace{0.5cm}

\end{abstract}
\end{center}

\section{Introduction}

Let $\Omega \subset R^2 $ be a bounded polygonal domain with boundary $\partial \Omega$. We
consider the following elliptic singular perturbation model:
\begin{equation}\label{pde}
\left\{
\begin{array}{rcc}
\begin{aligned}
\epsilon^2\Delta^2{u}-\Delta{u}&=f   \qquad \text{in} \quad\Omega \\
u=0,\quad \frac{\partial u}{\partial n}&=0 \qquad \text{on} \:\:\partial\Omega \\
\end{aligned}
\end{array}
\right.
\end{equation}
Here $f\in L^2(\Omega)$,  $\Delta$ is the standard  Laplace operator,  $\partial / \partial n $ denotes the normal
derivative on $\partial \Omega$,\, $n=(n_1,n_2)^T$ is the unit outer normal vector of $\partial
\Omega$, and $\epsilon$ is a real parameter such that 0 $< \epsilon \leqslant$ 1. It is obvious
that the equation degenerates to Poisson's equation when $\epsilon$ tends to zero.

For fourth-order elliptic problems, conforming finite element methods require $C^1$ continuity.
This usually leads to complicated element construction(see, e.g. \cite{Ciarlet1978}). In order
to overcome the $C^1$ difficulty, nonconforming finite element methods are often preferred.
Among the existing nonconforming  elements, the triangular  Morley element is the simplest one
\cite{Morley}. For the convergence analysis of this element, one can see \cite{Ciarlet1978, LL,
Shi, Strang}. But, as shown in \cite{NTR,Wang2001}, the Morley method is not uniformly
convergent with respect to the perturbation  parameter $\epsilon$ for the singular perturbation
problem (\ref{pde}). It  diverges for Poisson's equations, i.e. the limit problem when
$\epsilon$ tends to zero.

In  \cite{NTR}, Nilssen, TAI  and Winther proposed a $C^0$ nonconforming  triangular element.
It has 9 degrees in each element and the function space contains complete polynomials of degree
2. Uniform  convergent rate in $\epsilon$ was deduced for the problem (\ref{pde}). Chen, Zhao
and Shi  presented and analyzed in  \cite{Chen;Zhao;Shi2005} a nine parameter triangular
element and a twelve parameter rectangular element with double set parameters. In
\cite{Wang;Xu;Hu2006}, Wang, Xu  and Hu  derived a modified   Morley element method. This
method  uses the triangular  Morley element or rectangular Morley element, but the linear or
bilinear approximation of finite element functions is used in the lower part of the bilinear
form.  It was shown that the modified scheme converges  uniformly in $\epsilon$.

In this paper we focus on the analysis of uniformly stable rectangular elements for the problem
(\ref{pde}). In Section 2 we introduce some notations and the weak formulations.  In Section 3
we discuss   general conditions for the construction of uniformly convergent nonconforming
finite elements. We show that  the rectangular Morley element is  uniformly convergent in
Section 3.  We also propose   a new uniformly  convergent $C^0$ higher order rectangular
element in Section 4.  Finally  we give some numerical examples in Section 5.

\section{Weak formulations}
 We first introduce some notations. For a non-negative integer
$m$, let $H^m( \Omega)$ denote  the usual Sobolev space  with norm $\| \cdot \|_m$ and
semi-norm $ | \cdot |_m$.    $ H^m_0( \Omega)$ denotes the closure of $C_0^\infty$ in $
H^m(\Omega)$. We have
$$ H^1_0( \Omega)=\{v \in H^1: v=0 \,\,on\,\,\partial \Omega \},$$
$$ H^2_0( \Omega)=\{v \in H^2\cap H^1_0: \frac{\partial v}{\partial n}=0 \,\,on\,\,\partial \Omega \}.$$
$H^{-m}( \Omega)$  denotes the dual space of $H^m_0( \Omega)$, and   $ L^2(\Omega)=H^0(\Omega)$
is the space of square-integrable functions with  the inner product $(\cdot \, ,\cdot )$.

We also use the notation $\| \cdot \|_{m,K}$ (or   $ | \cdot |_m$)   to indicate that the norm (or semi-norm)
is defined with respect to a domain $K$.

Let $\mathcal {T}_h$ be a shape regular triangulation of the domain $\Omega$ with the mesh
parameter $h=\max_{T\in \mathcal {T}_h}\{\mbox{diameter of $T$}\}$.  Let $e$ denote any edge of
an element and $\mathcal {E}_h$ be the set of all interior edges in $\mathcal {T}_h$.  We use
$[v]$ to denote the jump of  a function $v$ across an interior edge $e$, and $[v]=v$ when $e
\subset \partial \Omega$.  We denote by $P_k$    the set of polynomials of degree $\leq k$,
and by   $Q_k$  the set of polynomials of degree $\leq k$ in each variable.

For simplicity, we use $X \lesssim (\gtrsim ) Y$ to denote that there exists a constant $C$,
independent of the mesh size $h$ and the perturbation parameter $\epsilon$, such that $X\leq\!(\geq)CY$.

Let  $D^2u$ denote  the 2$\times$2-tensor of second order partials of $u$ with
$(D^2u)_{i,j}=\partial^2u/\partial x_i\partial x_j$, with  the scalar product of tensors
defined by $$D^2u:D^2v=\sum_{i,j=1}^2 \frac{\partial^2u}{\partial x_i\partial x_j}
\frac{\partial^2v}{\partial x_i\partial x_j}.$$ It is easy to verify that
$$ \int_\Omega D^2u:D^2v =\int_\Omega \Delta u \,\Delta v,\ \ \forall u,v\in H^2_0(\Omega).$$
Define the bilinear forms
$$ a(u,v)=\int_\Omega D^2u:D^2v ,\ \ \ \ b(u,v)=\int_\Omega \nabla u \cdot \nabla v .$$
Then the weak form of problem (\ref{pde}) reads as : Given $f\in L^2(\Omega)$, find $u\in H^2_0( \Omega)$
such that
\begin{equation}\label{continuous weak formula}
\epsilon^2a(u,v)+b(u,v)=(f,v),\, \, \forall v\in H^2_0( \Omega).
\end{equation}
We also define an energy norm on $H^2(\Omega)$ relative to the parameter $\epsilon$ as follows:
$$\|u\|_\epsilon=\sqrt{\epsilon^2a(u,u)+b(u,u)}.$$

Let $V_h \not\subset H^2_0( \Omega)$ be a finite-dimensional space. Define the bilinear forms on
$V_h$ by
$$a_h(u_h,v_h)=\sum \limits_{T \in \mathcal {T}_h} \int_T D^2u_h:D^2v_h, \ \ \
b_h(u_h,v_h)= \sum \limits_{T \in \mathcal {T}_h} \int_T \nabla u_h \cdot \nabla v_h.$$ The
discrete weak formulation corresponding to the problem (\ref{continuous weak formula}) reads
as:

Find $u_h \in V_h$ such that
\begin{equation}\label{discrete weak formula}
\epsilon^2 a_h(u_h,v_h)+b_h(u_h,v_h)=(f,   v_h) \quad \forall v_h \in V_{h}.
\end{equation}

\begin{remark}\label{modified}
In the case that $f\in H^{-1}(\Omega)$ and $V_h \not\subset H^1_0( \Omega)$,  one can use the
following modified scheme to replace (\ref{discrete weak formula}):
\begin{equation}\label{discrete weak formula2}
\epsilon^2 a_h(u_h,v_h)+b_h(u_h,v_h)=(f, \emph{$\Pi$} v_h) \quad \forall v_h \in V_{h}.
\end{equation}
Here $  \Pi:V_h\longrightarrow V_h\bigcap H^1_0( \Omega)$ is an operator which preserves linear
polynomials locally. The error analysis is almost the same as that for the scheme (\ref{discrete weak formula}).
\end{remark}

Similar to the continuous level,  we define the discrete norm on $V_h$ as
$$\|u\|_{\epsilon,h}=\sqrt{\epsilon^2a_h(u,u)+b_h(u,u)}.$$
%

\section{Assumptions for element construction and general convergence results}

Let us first make some assumptions on the finite element space $V_h$.
$$\forall\:T\in \mathcal {T}_h,  P_2(T) \subset V_h|_T, \mbox{ and } I_h|_T\:v= v,\  \forall v\in P_2(T), \leqno(H1)$$
where $I_h$ is the interpolate operator  associated with $V_h$;
$$\begin{array}{l}
\forall \:v_h \in V_h, \ v_h  \mbox{ is continuous at the vertices of elements}\\
\hskip2cm\mbox{ and is zero at
the  vertices  on } \partial\Omega;\end{array}
\leqno(H2)
$$
$$\begin{array}{l}
\forall \:v_h \in V_h,\ \int_e \frac{\partial v_h}{\partial n}\, ds  \mbox{  is continuous
across the element edge } e\\
\hskip2cm\mbox{ and is zero on }e \subset \partial\Omega ;\end{array}
\leqno(H3)
$$

From (H1) and (H2) we easily know that there exists an interpolation operator
$ \Pi:V_h\longrightarrow V_h\bigcap H^1_0( \Omega)$  such that
\begin{equation}\label{interpolation}
\Pi|_Tv=v,\ \ \ \forall v\in P_1(T),\ \  \forall T\in \mathcal {T}_h.
\end{equation}
In fact,  $\Pi$ can be taken as the interpolation operator corresponding to the continuous linear element
when $T$ is a triangle or to the continuous bilinear element when $T$ is a rectangle. Especially,
when $V_h\subset H^1_0(\Omega)$, we can take
\begin{equation}\label{identity}
\Pi={\bf I},
\end{equation}
where ${\bf I}:V_h\longrightarrow V_h$ is the identity operator.

If $V_h\not\subset H^1_0(\Omega)$,  we further assume that
$$ \int_T \nabla(v_h-\Pi v_h)=0, \ \forall \:v_h \in V_h, \ \  \forall T\in \mathcal {T}_h.  \leqno(H4)$$


The assumption (H2) ensures that $\|\cdot\|_{\epsilon, h}$ is a norm on $V_h$. This guarantees
the existence and uniqueness of the solution, $u_h\in V_h$, to the problem (\ref{discrete weak formula}).

To estimate the error $u-u_h$ in the energy norm, we need
the second Strang lemma (see \cite{Ciarlet1978}, Theorem 4.2.2):
\begin{lem}\label{strang lemma}
Let $u$ and $u_h$ be the  solutions to the problems (\ref{continuous weak formula}) and (\ref{discrete
weak formula}) respectively. Then it holds 
\begin{eqnarray}\label{second Strang lemma 1}
\|u-u_h\|_{\epsilon,h}\lesssim \inf \limits_{v_h \in V_{h}}\|u-v_h\|_ {\epsilon,h}+\sup
\limits_{w_h \in V_h,\:w_h\neq0} \frac{E_{\epsilon,h}(u,w_h)}{\|w_h\|_ {\epsilon,h}}
\end{eqnarray}
where
\begin{eqnarray*}
\begin{aligned}
E_{\epsilon,h}(u,w_h)&=\epsilon^2a_h(u,w_h)+b_h(u,w_h)-(f, w_h)
\end{aligned}
\end{eqnarray*}
\end{lem}

From the assumption (H1) and the standard  interpolation theory, we can have
\begin{equation}\label{interpolation error}
\sum \limits_{T \in \mathcal{T}_h}\|v-I_hv\|_{j,T}\lesssim h^{k-j}|v|_k \quad \forall v \in
H^2_0\cap H ^k, \quad j=0,1,2,\ \  k=2,3.
\end{equation}
Thus  the   approximation error term, i.e. the first term on the right of (\ref{second Strang lemma 1}),
can be bounded as
\begin{equation}\label{approximation error}
 \inf \limits_{v_h \in V_{h}}\|u-v_h\|_ {\epsilon,h}\lesssim \|u-I_hu\|_{\epsilon,h}\lesssim\left\{\begin{array}{l}
 h(|u|_2+\epsilon|u|_3)\\
 h(h+\epsilon)|u|_3.\end{array} \right.
\end{equation}

As for the  consistency error term $E_{\epsilon,h}(u,w_h)$,   we can   express it as
\begin{eqnarray}\label{consist1}
 \begin{aligned}
   E_{\epsilon,h}(u,w_h)&=\sum \limits_{T \in \mathcal {T}_h} \int_T (\epsilon^2 D^2u:D^2w_h+
    \nabla u \cdot \nabla w_h)-\int_\Omega ( \epsilon^2\triangle^2u-\triangle u) \Pi w_h\\
    &\hskip2cm +\int_\Omega f\ ( \Pi w_h-w_h)\\
    &=\epsilon^2 \left(\sum \limits_{T \in \mathcal {T}_h} \int_T D^2u:D^2w_h-
    \int_\Omega  \triangle^2u \:\Pi w_h\right)\\
    &\ \ +\left(\sum_{T \in \mathcal {T}_h}\int_T \nabla u \cdot \nabla w_h+\int_\Omega \triangle u \:\Pi w_h\right)+\int_\Omega f\ ( \Pi w_h-w_h)\\
\end{aligned}
\end{eqnarray}
From Green's formula and the fact $\Pi w_h|_{\partial \Omega}=0$, we have
\begin{eqnarray*}
\begin{aligned}
\int_T D^2u:D^2w_h&=\int_T \Delta u \Delta
w_h+\int_T (2\partial_{12}u\partial_{12}w_h-\partial_{11}u\partial_{22}w_h-\partial_{22}u\partial_{11}w_h)\\
&=\int_T \Delta u \Delta
w_h+\int_{\partial T}(-\frac{\partial^2 u}{\partial^2 s}\frac{\partial w_h}{\partial
n}+\frac{\partial^2 u}{\partial n \partial s}\frac{\partial w_h}{\partial s})\\
&=\int_{\partial T}\Delta u\frac{\partial w_h}{\partial n} -
\int_T \nabla(\Delta u)\cdot \nabla w_h+\int_{\partial T}(-\frac{\partial^2 u}{\partial^2 s}\frac{\partial w_h}{\partial
n}+\frac{\partial^2 u}{\partial n \partial s}\frac{\partial w_h}{\partial s}),
\end{aligned}
\end{eqnarray*}
$$\int_\Omega \Delta^2 u \Pi w_h=\int_{\partial \Omega}\frac{\partial
(\Delta u)}{\partial n}\Pi w_h - \int_{\Omega} \nabla(\Delta u)\cdot \nabla \Pi w_h=- \int_{\Omega} \nabla(\Delta u)\cdot \nabla \Pi w_h,$$
and
$$ \int_\Omega \triangle u \:\Pi w_h=-\int_\Omega \nabla u\cdot \nabla\Pi w_h+\int_{\partial \Omega}\frac{\partial
 u}{\partial n}\Pi w_h=-\sum_{T \in \mathcal {T}_h}\int_T \nabla u \cdot \nabla \Pi w_h,$$
where in the first equality $s$
denotes the unit tangential vector along $\partial T$.
 The above three relations,  together with (\ref{consist1}), imply
\begin{eqnarray}\label{consist}
 \begin{aligned}
   E_{\epsilon,h}(u,w_h)&=\epsilon^2 \sum_{T \in \mathcal {T}_h} \int_{\partial T}
    \{(\Delta u-\frac{\partial^2 u}{\partial^2 s}) \frac{\partial w_h}{\partial n}+
    \frac{\partial^2 u}{\partial n \partial s}\frac{\partial w_h}{\partial s}\} \\
    &\hskip5mm+\epsilon^2 \sum_{T \in \mathcal {T}_h}\int_T \nabla(\Delta u)\cdot \nabla(\Pi w_h- w_h) \\
    &\hskip5mm+\sum_{T \in \mathcal {T}_h} \int_T \nabla u \cdot \nabla(w_h-\Pi w_h)+\int_\Omega f\ ( \Pi w_h-w_h)\\
    &=:J_1+J_2+J_3+J_4,
\end{aligned}
\end{eqnarray}

\begin{remark}\label{J234}
When $V_h \subset H^1_0( \Omega)$,  from (\ref{identity}) we have $\Pi w_h=w_h$, which implies
$J_2= J_3=J_4=0.$
Then the consistency term is reduced to
\begin{eqnarray}\label{consistency}
\begin{aligned}
E_{\epsilon,h}(u,w_h)&=J_1=\epsilon^2 \sum_{T \in \mathcal {T}_h} \int_{\partial T}\{(\Delta
u-\frac{\partial^2 u}{\partial^2 s}) \frac{\partial w_h}{\partial n}+\frac{\partial^2 u}{\partial n
\partial s}\frac{\partial w_h}{\partial s}\}d s.
\end{aligned}
\end{eqnarray}
\end{remark}

For the consistency error term $E_{\epsilon,h}(u,w_h)$, we have the following conclusion:

\begin{lem}\label{consistencylemma}
Under the conditions of Lemma \ref{strang lemma} and the assumptions (H1)-(H4),  it holds
\begin{eqnarray}\label{consistencyinequality1}
E_{\epsilon,h}(u,w_h) \lesssim h(|u|_2+\epsilon|u|_3+||f||_0)\|w_h\|_{\epsilon,h}
\end{eqnarray}
Furthermore, if $V_h \subset H_0^{1}(\Omega)$, then
\begin{eqnarray}\label{consistencyinequality2}
E_{\epsilon,h}(u,w_h)\lesssim \left\{\begin{array}{l}
h\epsilon |u|_3\|w_h\|_{\epsilon,h},\\\\
h^{1/2}\epsilon |u|_2^{1/2} |u|_3^{1/2}\|w_h\|_{\epsilon,h}.
\end{array}\right.
\end{eqnarray}
\end{lem}

\begin{proof}
From the assumptions $(H2)$ and $(H3)$,  we have
$$\int_e [\frac{\partial w_h}{\partial s}]\,
ds=0,\ \ \ \int_e [\frac{\partial w_h}{\partial n}]\, ds=0.$$
By  a standard scaling argument (see, for example,  \cite{Brenner1994}, Pages 205-207),  it holds
\begin{eqnarray}\label{consistencyinequality3}
J_1=\epsilon^2 \sum_{e \in \mathcal {E}_h} \int_e\{(\Delta u-\frac{\partial^2 u}{\partial^2 s})
[\frac{\partial w_h}{\partial n}]+\frac{\partial^2 u}{\partial n
\partial s}[\frac{\partial w_h}{\partial s}]\}d s\lesssim \left\{\begin{array}{l}
h\epsilon |u|_3\|w_h\|_{\epsilon,h},\\\\
h^{1/2}\epsilon |u|_2^{1/2} |u|_3^{1/2}\|w_h\|_{\epsilon,h}.
\end{array}\right.
\end{eqnarray}

When $V_h \subset H_0^{1}(\Omega)$, the above inequality (\ref{consistencyinequality3}),
together with Remark \ref{J234}, indicates (\ref{consistencyinequality2}).

When $V_h\not \subset H_0^{1}(\Omega)$,  from Schwarz's inequality, (5) and the standard interpolation
theory, we have
\begin{eqnarray}\label{consistencyinequality4}
J_2=\epsilon^2 \sum_{T \in \mathcal {T}_h}\int_T \nabla(\Delta u)\cdot \nabla(\Pi w_h- w_h) \lesssim
\left\{\begin{array}{l}
h\epsilon |u|_3\|w_h\|_{\epsilon,h},\\\\
h^{1/2}\epsilon^{3/2} |u|_3\|w_h\|_{\epsilon,h}.
\end{array}\right.
\end{eqnarray}
Similarly, by (H4) we obtain
\begin{eqnarray}\label{consistencyinequality5}
\begin{aligned}
J_3&=\sum \limits_{T\in\mathcal {T}_h}\int_T \nabla u \cdot \nabla(w_h-\Pi w_h)\\
&=\sum \limits_{T \in\mathcal{T}_h}\int_T (\nabla u-\Pi_0 \nabla u) \cdot \nabla(w_h-\Pi w_h)\\
&  \lesssim h
|u|_2\|w_h\|_{\epsilon,h},
\end{aligned}
\end{eqnarray}
where  $\Pi_0 \nabla u=\frac{1}{|T|}\int_T \nabla u $.  We also have
\begin{eqnarray}\label{consistencyinequality6}
J_4=\int_\Omega f\ ( \Pi w_h-w_h)
  \lesssim h
||f||_0\|w_h\|_{\epsilon,h}.
\end{eqnarray}

As a result, the estimation  (\ref{consistencyinequality1}) follows from   (\ref{consist}),
(\ref{consistencyinequality3})-(\ref{consistencyinequality6}).
\end{proof}

From Lemma \ref{strang lemma}, the estimation (\ref{approximation error}), and Lemma
\ref{consistencylemma},  we immediately get the following main convergence result:
\begin{thm}\label{theorem1}
Suppose (H1)-(H4) hold true. Let $u$ and $u_h$ be the  solutions to the problems (\ref{continuous weak
formula}) and (\ref{discrete weak formula}) respectively.  Then it holds
\begin{equation}\label{theorem1inequality1}
\|u-u_h\|_{\epsilon,h}\lesssim h(|u|_2+\epsilon|u|_3+||f||_0).
\end{equation}
Moreover, when  $V_h \subset H_0^{1}(\Omega)$, we have
\begin{equation}\label{theorem1inequality2}
\|u-u_h\|_{\epsilon,h}\lesssim h(h+\epsilon )|u|_3.
\end{equation}
\end{thm}

In next section,  on the basis of (H1)-(H4), we will analyze the rectangular Morley element and
construct  a new high order nonconforming rectangular element for the problem .

\numberwithin{lem}{section} \numberwithin{thm}{section} \numberwithin{remark}{section}
\section{Nonconforming rectangular   elements}
In what follows we assume the domain $\Omega$ is a bounded polygonal domain   with a shape-regular rectangular
mesh subdivision  $\mathcal {T}_h$.
\subsection{Rectangular Morley element}
Given a  rectangle  $T\in \mathcal {T}_h$ with center $a_0=(x_0,y_0)$,
\begin{equation}\label{rectangle}
T=\{(x,y)|\, x=x_0+h_1\xi, y=y_0+h_2\eta ,-1 \leq \xi\leq 1, -1 \leq \eta \leq 1 \}.
\end{equation}
 Let $a_i $ and $e_i$  $(i=1,2,3,4)$ be its vertices and  edges, respectively (see Figure \ref{rectmorley}),  with
 edge lengthes $|e_1|=|e_2|=2h_1$ and $|e_3|=|e_4|=2h_2$.  The  rectangular Morley element is then described by
 $(T,P_T,N_T)$ \cite{ZhangWang1991}:
\par \text{(1)} \qquad $P_T=P_2(T)+span\{x^3,y^3\}$;
\par \text{(2)} \qquad For $  \forall v\in C^1(T)$,  the set of degrees of freedom
$$N_T(v)=\left\{v(a_i), \frac{1}{|e_i|}\int_{e_i}\frac{\partial v}{\partial n} ds:\ \ 1\leq i\leq 4\right\}.$$
 \begin{center}
\setlength{\unitlength}{0.9cm}
\begin{picture}(6,4.5)\label{rectmorley}
\put(0,0.5){\line(1,0){5}}        \put(5,0.5){\line(0,1){3}} \put(0,0.5){\line(0,1){3}}
\put(0,3.5){\line(1,0){5}} \put(0,0.5){\circle*{0.15}}       \put(5,0.5){\circle*{0.15}}
\put(0,3.5){\circle*{0.15}}       \put(5,3.5){\circle*{0.15}} \put(5,2){\vector(1,0){0.5}}
\put(0,2){\vector(-1,0){0.5}} \put(2.5,0.5){\vector(0,-1){0.5}} \put(2.5,3.5){\vector(0,1){0.5}}
\put(-0.3,0.2){\bf{$a_1$}}        \put(5,0.2){\bf{$a_2$}} \put(5,3.7){\bf{$a_3$}}
\put(-0.2,3.7){\bf{$a_4$}} \put(5.1,1.6){\bf{$e_1$}}         \put(-0.4,1.6){\bf{$e_2$}}
\put(2.7,3.7){\bf{$e_3$}}         \put(2.7,0.1){\bf{$e_4$}} \put(-3,-0.5){Figure \ref{rectmorley}:
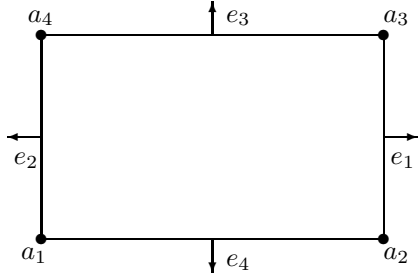
The element diagram of the rectangular Morley element}
\end{picture}
\end{center}
\vspace{0.4cm}

\begin{remark}\label{equivalent DOF} In fact, it is easy to know that the element $(T,P_T,N_T)$ is
interpolation-equivalent to $(T,P_T,N_T')$, where
$$N_T'(v)=\left\{v(a_i), \frac{\partial v}{\partial n}(b_i): \ \ 1\leq i\leq 4\right\},$$
and $b_i$ is  the  midpoint of the   edge $e_i$ for $1\leq i \leq 4 $.
\end{remark}

Define the Morley space by
\begin{eqnarray*}\label{Morley space}
M_h=\{v_h\in L^2(\Omega): &v_h|_T\in P_T,\ \forall T\in \mathcal {T}_h; \ v_h \mbox{ is continuous at  the   vertices}\\
&\mbox{of  elements and vanishes at the vertices on } \partial \Omega;\\
& \int_e \frac{\partial v_h}{\partial n}\, ds  \mbox{  is continuous
across the element edge } e\\
&\mbox{and vanishes on }e \subset \partial\Omega \}.
\end{eqnarray*}
Obviously we have $M_h\not\subset H^1_0(\Omega)$.
For $\forall v_h \in M_h, \ T\in \mathcal
{T}_h$, we can write it in the form
\begin{equation}\label{vh}
v_h|_T=\sum \limits_{i=1}^{4}v_h(a_i)\ q_i+\sum
\limits_{i=1}^{4}\frac{1}{|e_i|}\int_{e_i}\frac{\partial v_h} {\partial n}ds\ q_{i+4},
\end{equation}
where $q_i \  ( 1 \leq i \leq 8)$ are the corresponding   basis functions given by
 \begin{eqnarray*}
\begin{aligned}
q_1(\xi,\eta)&=\frac{1}{4}\,(1-\xi)\,(1-\eta)+\frac{1}{8}\,\xi\,(\xi^2-1)+\frac{1}{8}\,\eta\,(\eta^2-1),\\
q_2(\xi,\eta)&=\frac{1}{4}\,(1+\xi)\,(1-\eta)-\frac{1}{8}\,\xi\,(\xi^2-1)+\frac{1}{8}\,\eta\,(\eta^2-1),\\
q_3(\xi,\eta)&=\frac{1}{4}\,(1+\xi)\,(1+\eta)-\frac{1}{8}\,\xi\,(\xi^2-1)-\frac{1}{8}\,\eta\,(\eta^2-1),\\
q_4(\xi,\eta)&=\frac{1}{4}\,(1-\xi)\,(1+\eta)+\frac{1}{8}\,\xi\,(\xi^2-1)-\frac{1}{8}\,\eta\,(\eta^2-1),\\
\end{aligned}
\end{eqnarray*}
\vspace{-0.5cm}
\begin{eqnarray*}
\begin{aligned}
q_5(\xi,\eta)&=\frac{h_1}{4}(\xi+1)^2\,(\xi-1),\quad
q_6(\xi,\eta)=-\frac{h_1}{4}(\xi+1)\,(\xi-1)^2,\\
q_7(\xi,\eta)&=\frac{h_2}{4}(\eta+1)^2\,(\eta-1),\quad
q_8(\xi,\eta)=-\frac{h_2}{4}(\eta+1)\,(\eta-1)^2.\\
\end{aligned}
\end{eqnarray*}

Now we take $V_h=M_h$, and let $\Pi$ be the usual bilinear interpolation
operator corresponding to the $H^1-$conforming bilinear element with respect to $\mathcal {T}_h$.

By the definition we easily know that the assumptions (H1)-(H3) hold for the Morley space $V_h$.
We will further show (H4) also holds true. We have
\begin{lem}
For $\forall \:v_h \in M_h,\ \forall T\,\in \mathcal {T}_h$,  it holds
$$\int_T \nabla(v_h-\Pi v_h) =0.$$
\end{lem}

\begin{proof}
For $v_h \in M_h$, the bilinear interpolation  $\Pi v_h$ can be expressed as:
\begin{equation*}
 \Pi v_h|_T=\sum \limits_{i=1}^{4}v_h(a_i) \tilde{p}_i,
\end{equation*}
where $ \tilde{p}_i \ ( 1 \leq i \leq 4)$ are the corresponding bilinear basis functions, namely
$$ \tilde{p}_1 =\frac{1}{4}\,(1-\xi)\,(1-\eta),\ \tilde{p}_2=\frac{1}{4}\,(1+\xi)\,(1-\eta),$$
$$  \tilde{p}_3=\frac{1}{4}\,(1+\xi)\,(1+\eta),\  \tilde{p}_4=\frac{1}{4}\,(1-\xi)\,(1+\eta).$$
 Let $\tilde{q_i}=q_i- \tilde{p}_i$,  then from (\ref{vh}) we have
\begin{eqnarray}\label{vh-Pivh}
v_h-\Pi v_h= \sum \limits_{i=1}^{4}v_h(a_i)\ \tilde{ q_i}+\sum \limits_{i=1}^{4}\frac{1}{|e_i|}
\int_{e_i}\frac{\partial v_h}{\partial n} ds\ q_{i+4}.
\end{eqnarray}
It is easy to see that
\begin{eqnarray*}
\int_{\hat{T}}\hat{\nabla} \tilde{q_i} d\xi d\eta=0, 1\leq i\leq4,
\end{eqnarray*}
\begin{eqnarray*}
 \int_{\hat{T}}\hat{\nabla}
q_i d\xi d\eta=0, 5\leq i\leq8,
\end{eqnarray*}
where $\hat T=[-1,1]\times[-1,1]$,
$\hat{\nabla}=[\frac{\partial}{\partial \xi}, \frac{\partial}{\partial \eta}]^T$.  These indicate
\begin{eqnarray*}
\begin{aligned}
\int_T \nabla \tilde{q_i} dxdy&=\left[
\begin{array}{c c}
h_2  &0   \\
0    &h_1 \\
\end{array}
\right ]\int_{\hat{T}}\hat{\nabla} \tilde{q_i} d\xi d\eta=0, \ 1\leq i\leq4,\\
\int_T \nabla q_i dxdy&= \left[
\begin{array}{c c}
h_2  &0   \\
0    &h_1 \\
\end{array}
\right ] \int_{\hat{T}}\hat{\nabla} q_i d\xi d\eta=0, \ 5\leq i\leq8.
\end{aligned}
\end{eqnarray*}
The above two relations, together with (\ref{vh-Pivh}), yield
\begin{equation*}
\int_T \nabla(v_h-\Pi v_h)=\int_T\left( \sum \limits_{i=1}^{4} v_h(a_i)\nabla\tilde{q_i}+
\sum \limits_{i=1}^{4}\frac{1}{|e_i|}\int_{e_i}\frac{\partial v_h}{\partial n}ds\
\nabla q_{i+4}\right)=0.
\end{equation*}
\end{proof}

As a result,  the rectangular Morley space $M_h$ satisfies the assumptions (H1)-(H4). Then,
from Theorem \ref{theorem1},  we have

\begin{thm}\label{Morley error}
Let $u$ and $u_h$ be the  solutions to the problems (\ref{continuous weak
formula}) and (\ref{discrete weak formula}) respectively.  Then, for the rectangular Morley element,
it holds the following error estimate:
\begin{equation}\label{Morleyerror}
\|u-u_h\|_{\epsilon,h}\lesssim h(|u|_2+\epsilon |u|_3+||f||_0).
\end{equation}
\end{thm}

In next subsection, we will propose an extended  high order $C^0$ rectangular  Morley element for the problem (\ref{pde}).
\subsection{Extended high order rectangular Morley  element}



Let  $T\in \mathcal {T}_h$ be a rectangle given by (\ref{rectangle}) . Let $m_i$ be the four midpoints of the edges,
$1 \leq i \leq 4$ (see  Figure \ref{rectangleelement}).  Introduce three functions like
\begin{equation}\label{phi123}
\left\{\begin{array}{l}
\phi_1(\xi,\eta)=\xi^4\,(1-\eta^2),\\
\phi_2(\xi,\eta)=\eta^3\,(1-\xi^2),\\
\phi_3(\xi,\eta)=(\xi+\eta)(1-\xi^2)\,(1-\eta^2).
\end{array}\right.
\end{equation}
Then the extended high order  rectangular Morley  element   $(T,Q_T,\Phi_T)$ is given by
\par \text{(1)} $Q_T=Q_2(T)+span \{\,\phi_1(\xi,\eta),\phi_2(\xi,\eta),\phi_3(\xi,\eta)\};$
\par \text{(2)} For $v \in C^1(T)$, the set of degrees of freedom
$$\Phi_T=\left\{v(a_i),\quad v(m_i),\quad\,\int_{e_i}\frac{\partial v}{\partial n}ds,\:1\leq i\leq 4\right\},$$
where $Q_2(T)$ is the set of bi-quadratic polynomials on $T$.

\begin{center}
\setlength{\unitlength}{1cm}
\begin{picture}(6,4.2)\label{rectangleelement}
\put(0,0.5){\line(1,0){5}}        \put(5,0.5){\line(0,1){3}} \put(0,0.5){\line(0,1){3}}
\put(0,3.5){\line(1,0){5}} \put(0,0.5){\circle*{0.15}}       \put(5,0.5){\circle*{0.15}}
\put(0,3.5){\circle*{0.15}}       \put(5,3.5){\circle*{0.15}} \put(5,2){\circle*{0.15}}
\put(0,2){\circle*{0.15}} \put(2.5,0.5){\circle*{0.15}}     \put(2.5,3.5){\circle*{0.15}}
\put(5,2){\vector(1,0){0.5}}      \put(0,2){\vector(-1,0){0.5}}
\put(2.5,0.5){\vector(0,-1){0.5}} \put(2.5,3.5){\vector(0,1){0.5}} \put(-0.3,0.2){\bf{$a_1$}}
\put(5,0.2){\bf{$a_2$}} \put(5,3.7){\bf{$a_3$}}           \put(-0.2,3.7){\bf{$a_4$}}
\put(5.1,1.6){\bf{$e_1$}}         \put(2.7,3.6){\bf{$e_2$}} \put(-0.5,1.6){\bf{$e_3$}}
\put(2.7,0.1){\bf{$e_4$}} \put(4.4,1.9){\bf{$m_1$}}         \put(2.3,3.1){\bf{$m_2$}}
\put(0.2,1.9){\bf{$m_3$}}         \put(2.3,0.7){\bf{$m_4$}} \put(-4,-0.5){Figure
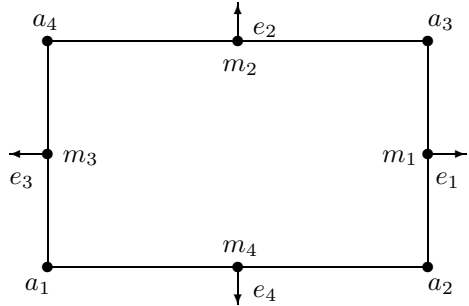
\ref{rectangleelement}: The element diagram of the extended high order  rectangular Morley
element}
\end{picture}
\end{center}
\vspace{0.2cm}

\begin{lem}\label{unisolventlemma}
For the extended high order  rectangular Morley element, $\Phi_T$ is $Q_T$-unisolvent.
\end{lem}

\begin{proof}
We only need to give the proof on  $T=\hat{T}=[-1,1]\times [-1,1]$.
Let $\{p_i :i=1,2,\cdots,9\} $ be a basis of $Q_2(T)$ defined by
\begin{equation*}
\begin{aligned}
p_1(\xi,\eta)&=\frac{1}{4}\xi\,\eta\,(1-\xi)\,(1-\eta), \quad
p_2(\xi,\eta)=-\frac{1}{4}\xi\,\eta\,(1+\xi)\,(1-\eta), \\
p_3(\xi,\eta)&=\frac{1}{4}\xi\,\eta\,(1+\xi)\,(1+\eta),  \quad
p_4(\xi,\eta)=-\frac{1}{4}\xi\,\eta\,(1-\xi)\,(1+\eta), \\
p_5(\xi,\eta)&=\frac{1}{2}(1-\eta^2)\,\xi\,(1+\xi),     \quad
p_6(\xi,\eta)=\frac{1}{2}(1-\xi^2)\,\eta\,(1+\eta),   \,\\
p_7(\xi,\eta)&=-\frac{1}{2}(1-\eta^2)\,\xi\,(1-\xi), \quad \,
p_8(\xi,\eta)=-\frac{1}{2}(1-\xi^2)\,\eta\,(1-\eta),  \\
p_9(\xi,\eta)&=(1-\xi^2)\,(1-\eta^2).
\end{aligned}
\end{equation*}
Then, for  any function $w \in
Q_T$, we can express it   in the form
\begin{equation}\label{w-form}
w=\sum \limits_{i=1}^{9}\beta_i p_i+\beta_{10}\phi_{1}+\beta_{11}\phi_{2}+\beta_{12}\phi_{3},
\end{equation}
where $\phi_i \ (i=1,2,3)$ are given by (\ref{phi123}), and  the parameters $\beta_i\in \Re$ ($i=1,2,\cdots,12$). In what follows we  will show that, if  the twelve degrees of freedom of $w$ vanish, i.e.
$$w(a_i) =w(m_i) = \int_{e_i}\frac{\partial w}{\partial n}ds=0,\ \ i=1,2,3,4,$$
then  $w=0$.

Since
$$p_i(a_j)=\delta_{i,j},\ \phi_k(a_j)=0,\ \ j=1,2,3,4;\ i=1,2,\cdots,9;\ k=1,2,3, $$
from (\ref{w-form}) we immediately have
\begin{equation}\label{unisolvent1}
\beta_i=0,\ \ 1\leq i \leq 4.
\end{equation}

From $w(m_i)=0$ for $i=1,2,3,4$, we get
$$ -\beta_5-\beta_{11}=0,\, \beta_6+\beta_{10}=0,\, \beta_7+\beta_{11}=0,\, -\beta_8+\beta_{10}=0.$$
These yield
\begin{equation}\label{unisolvent2}
\beta_5=-\beta_{11},\,
\beta_6=-\beta_{10},\,\beta_7=-\beta_{11},\,\beta_8=\beta_{10}.
\end{equation}
 Hence, by (\ref{unisolvent1}) and  (\ref{unisolvent2}),  $w$  has  the  form
$$w=\beta_9 p_9+\beta_{10}(-p_6+p_8+\phi_1)+\beta_{11}(-p_5-p_7+\phi_2)+\beta_{12}\phi_3.$$
Finally, by  $\int_{e_i}\frac{\partial w}{\partial n}ds=0$ for $i=1,2,3,4$, we obtain the following system
\[
\left[
\begin{array}{r r r r}
-8/3  &8/3   &0     &-8/3  \\
-8/3  &8/15  &8/3   &-8/3  \\
-8/3  &8/3   &0     &8/3   \\
-8/3  &8/15  &-8/3  &8/3  \\
\end{array}
\right ]\left[\begin{array}{l}
\beta_9\\ \beta_{10}\\ \beta_{11}\\ \beta_{12}
\end{array}\right]={\bf 0}.
\]
  It is easy to know the solution to   this system is
\begin{equation}\label{unisolvent3}
\beta_9=\beta_{10}=\beta_{11}=\beta_{12}=0.
\end{equation}
Consequently, the desired conclusion follows from (\ref{unisolvent1})  (\ref{unisolvent2}) and (\ref{unisolvent3}).
\end{proof}

\begin{remark}
In fact, the selection of $\{ \phi_1(\xi,\eta),\phi_2(\xi,\eta),\phi_3(\xi,\eta) \}$ in the shape function space $\Phi_T$ is not unique. To ensure the element is $C^0$,   $\phi_i$ can be of the following
form:
\begin{equation*}
\left\{\begin{array}{l}
\phi_1(\xi,\eta)=g_1(\xi)\,(1-\eta^2),\\
\phi_2(\xi,\eta)=g_2(\eta)\,(1-\xi^2),\\
\phi_3(\xi,\eta)=g_3(\xi,\eta)\,(1-\xi^2)\,(1-\eta^2),
\end{array}\right.
\end{equation*}
where  $g_1(\xi)$ and $g_2(\eta)$ are polynomials of degrees $\geq3$, and $g_3(\xi,\eta)$ is a polynomial of degree $\geq 1$. For
example, a choice of $\phi_i$ different from  (\ref{phi123}) can be like
$$
\phi_1=\xi^3\,(1-\eta^2),\, \phi_2=\eta^4\,(1-\xi^2),\,\phi_3=(\xi+\eta)\,(1-\xi^2)\,(1-\eta^2).$$
\end{remark}

For $v\in Q_T$, it is easy to see that $v|_{e_i}\in P_2(e_i) \ (i=1,2,3,4)$. Then $v|_{e_i} $ is determined by   the
degrees of freedom associated to the endpoints and mid-point of the edge $e_i$. Therefore, the extended high order rectangular element $(T, Q_T,\Phi_T)$ is $C^0$.

Now we   define the extended high order  rectangular Morley   space as :
\begin{eqnarray*}\label{extended Morley space}
M^E_h=\{v_h\in L^2(\Omega): &v_h|_T\in Q_T,\ \forall T\in \mathcal {T}_h; \ v_h \mbox{ is continuous at     vertices and}\\
&\mbox{edge midpoints of  elements and vanishes at the vertices }\\
& \mbox{and edge midpoints  on } \partial \Omega;\ \int_e \frac{\partial v_h}{\partial n}\, ds  \mbox{  is continuous}\\
&\mbox{across the element edge } e
\mbox{ and vanishes on }e \subset \partial\Omega \}.
\end{eqnarray*}
It is obvious  $M^E_h\subset H_0^1(\Omega)$ and $M^E_h\not\subset
H^2(\Omega)$. Then this extended high order  rectangular Morley element  leads to a nonconforming
method for the fourth order problem.

Taking $V_h=M_h^E$ in (\ref{discrete weak formula}),  we 
easily know the assumptions (H1)-(H3) hold true. Then, from Theorem \ref{theorem1},
we have
\begin{thm}\label{extended Morley error}
 Let $u$ and $u_h$ be the  solutions to the problems (\ref{continuous weak
formula}) and (\ref{discrete weak formula}) respectively.  Then, for the  extended high order  rectangular Morley element, it holds
 \begin{equation}\label{theorem1inequality2}
\|u-u_h\|_{\epsilon,h}\lesssim \left\{
\begin{array}{l}
h(|u|_2+\epsilon |u|_3+||f||_0), \\
(h^2+\epsilon h)|u|_3.
\end{array}
\right.
\end{equation}
\end{thm}

\subsection{Boundary layers and uniform error estimates }

From Theorem \ref{Morley error} and Theorem \ref{extended Morley error}, we can conclude that the rectangular Morley element and the
extended high order  rectangular Morley element ensure linear convergence with respect to $h$,
uniformly in $\epsilon$, under the condition the semi-norm $|u|_2+\epsilon|u|_3$ being uniformly
bounded. In general, we can't expect that the norm $|u|_2$ and $|u|_3$ is bounded independent of
$\epsilon$. Actually, as $\epsilon$ approaches to zero $|u|_2$ and $|u|_3$ should be expected to blow
up. Hence, the convergence estimates given in the theorems   will deteriorate as $\epsilon$
becomes small. The purpose of this section is to establish error estimates which are
uniform with respect to the perturbation parameter $\epsilon \in [0,1]$ for  the rectangular Morley element and the extended high order
rectangular Morley element.

From the regularity theory for elliptic problems in non-smooth domains (see \cite{P.Gris}: Corollary
7.3.2.5), we have the following regularity result  for the problem    (\ref{pde}): If $f\in H^{-1}( \Omega)$ and $\Omega$ is convex, then  $u \in  H^3(\Omega)$ and it holds
\begin{equation}\label{regularity}
\|u\|_3 \leq C_\epsilon \parallel \!f \!\parallel  _{-1}.
\end{equation}
Here $C_\epsilon$ is a positive constant independent of $f$ but in general dependent on the parameter $\epsilon$.

In \cite{NTR}, Nilssen, Tai and Winther derived  the following refined regularity result:
\begin{lem}\label{uniformlemma}
Assume $f\in L^2( \Omega)$  and $\Omega$ is convex. Let $u =u^\epsilon\in H^2_{0}( \Omega) \cap H^3( \Omega) $ and  $u^0
\in H^1_{0}( \Omega) \cap H^2( \Omega) $ be respectively  the weak solutions to the problem(\ref{pde}) and  the reduced problem
\begin{equation}\label{lem4inequality1}
\left \{
\begin{aligned}
-\Delta u^0 &=f  \quad \text{in} \quad \Omega\\
u^0 &=0   \quad \text{on} \quad \partial \Omega.
\end{aligned}
\right.
\end{equation}
Then it holds
\begin{equation}\label{lem4inequality2}
\epsilon^{-1/2}|u-u^0|_1+\epsilon^{1/2}|u|_2+\epsilon^{3/2}|u|_3\lesssim\|f\|_0.
\end{equation}
\end{lem}

By this lemma, we have the following uniform result:
\begin{thm}\label{theorem4}
 Let $u$ and $u_h$ be the  solutions to the problems (\ref{continuous weak
formula}) and (\ref{discrete weak formula}) respectively. Assume the assumptions (H1)-(H4) hold true. Then it holds the following uniform error estimate
\begin{equation}\label{uniformestimate}
\parallel \!u-u_h \!\parallel_{\epsilon,h} \lesssim h^{1/2}\parallel \!f\!\parallel_0.
\end{equation}
\end{thm}

\begin{proof}
By the interpolation estimates (\ref{interpolation error}) and the regularity result (\ref{lem4inequality2}), we obtain
$$
\epsilon \|u- I_h u\|_2 \lesssim\epsilon |u|_2^{1/2}\|u- I_h u\|_2^{1/2} \lesssim\epsilon
h^{1/2}|u|_2^{1/2}|u|_3^{1/2}
\lesssim h^{1/2}\|f\|_0$$
and
\begin{equation*}
\begin{aligned}
\|u- I_h u\|_1  & \lesssim \|u-u^0-I_h(u-u^0)\|_1+\|u^0-I_h u^0\|_1\\
&\lesssim h^{1/2}(\epsilon^{-1/2} |u-u^0|_1)^{1/2}(\epsilon^{1/2}|u-u^0|_2)^{1/2}+ h |u^0|_2\\
& \lesssim h^{1/2}\|f\|_0.
\end{aligned}
\end{equation*}
These two inequalities yield   the  estimate of the approximation term,
\begin{equation}\label{thm4inequality2}
\inf \limits_{v \in V_{h}}\|u-v_h\|_{\epsilon,h}\lesssim \|u-I_h u\|_{\epsilon,h} \lesssim
h^{1/2}\|f\|_0.
\end{equation}
By Lemma \ref{strang lemma}, the only thing left is to estimate   the consistency error $E_{\epsilon,h}(u,w_h)=J_1+J_2+J_3+J_4$, where $J_i$ are defined in (\ref{consist}).  From (\ref{consistencyinequality3}), (\ref{consistencyinequality4}) and (\ref{lem4inequality2}), there hold
\begin{equation}\label{thm4inequality3}
J_1 \lesssim h^{1/2}\epsilon |u|_2^{1/2}|u|_3^{1/2}\|w_h\|_{\epsilon,h} \lesssim
h^{1/2}\|f\|_0\|w_h\|_{\epsilon,h},
\end{equation}
\begin{equation}\label{thm4inequality4}
J_2\lesssim h^{1/2}\epsilon^{3/2}|u|_3\|w_h\|_{\epsilon,h} \lesssim
h^{1/2}\|f\|_0\|w_h\|_{\epsilon,h}.
\end{equation}
For the term $J_3$, by the assumption (H4), standard interpolation theory and  (\ref{lem4inequality2}), we have
\begin{equation}\label{thm4inequality5}
\begin{aligned}
J_3&=\sum_{T \in \mathcal {T}_h} \int_T \nabla(u-u^0)
\cdot\nabla(w_h-\Pi w_h)+\sum_{T \in\mathcal{T}_h} \int_T (\nabla u^0-\Pi_0 \nabla u^0) \cdot \nabla(w_h-\Pi w_h) \\
&\lesssim h^{1/2}\!\!\sum_{T \in \mathcal {T}_h} |u-u^0|_{1,T}|w_h|_{1,T}^{1/2}|w_h|_{2,T}^{1/2}+ h\sum_{T \in \mathcal {T}_h}  |u^0|_{2,T} |w_h|_{1,T}\\
&\lesssim h^{1/2}\epsilon^{-1/2}|u-u^0|_1\|w_h\|_{\epsilon,h}+h|u^0|_2\|w_h\|_{\epsilon,h}\\
&\lesssim
h^{1/2}\|f\|_0\|w_h\|_{\epsilon,h},
\end{aligned}
\end{equation}
where the operator $\Pi_0$ is the same as in (\ref{consistencyinequality5}).
Finally,  the above estimates (\ref{thm4inequality2})-(\ref {thm4inequality5}), together with  (\ref{consistencyinequality6}), indicate   the desired uniform estimate
(\ref{uniformestimate}).
\end{proof}

\begin{coro}\label{morleyuniformlemma}
The rectangular Morley element and the extended high order rectangular Morley element are uniformly
convergent when applied to the problem (\ref{pde}), in a sense that   the uniform error estimate (\ref{uniformestimate}) holds true.
\end{coro}

\section{Numerical results}
In this section, we will show some numerical results of  the rectangular Morley element
and the extended high order rectangular Morley element.
\subsection{An example without boundary layers}

Let $\Omega=[0,1]\times [0,1]$ and $u(x_1,x_2)=\sin(\pi x_1)^2\sin(\pi x_2)^2$. For $\epsilon \geq 0$,
set $f=\epsilon^2\Delta^2{u}-\Delta{u}$. Then $u$ is the solution to the  problem (\ref{pde}) when
$\epsilon > 0$. The domain $\Omega$ is divided into $n^2$ squares of size $h\times h$, with $h=1/n$.

In   tables 1-2 we have listed the relative error in the energy norm, $\|u-u_h
\|_{\epsilon,h} / \|u\|_{\epsilon,h}$ for different values of $\epsilon, h$. For  comparison we
also consider the case $\epsilon=0$, i.e, the Poisson's problem with Dirichlet boundary conditions,
and the biharmonic problem $\Delta^2{u}=f$.

From the numerical results we can conclude that the rectangular Morley element and the extended high order
rectangular Morley element both converge for all $\epsilon \in [0,1]$. More precisely, for the
extended high order rectangular Morley element, the results show that relative energy error is
linear with respect to $h$ when $\epsilon$ is large while it is quadratic when  $\epsilon$ is small.
But the rectangular Morley element can only ensures linear convergence rate.  These are conformable to
  our theoretical results (\ref{Morleyerror}) and (\ref{theorem1inequality2}).

\subsection{An example with boundary layers}
We consider an example to verify the theoretical analysis for boundary layers.  Let  $\Omega=[0,1]\times [0,1]$ and
$u(x_1,x_2)=\epsilon(e^{-x_1/\epsilon}+e^{-x_2/\epsilon})-x_1^2x_2$, $f=2x_2$, and we assume the
Dirchlet and Neumann boundary condition  holds.

We   computed the relative error in the energy norm for various values of $\epsilon $ and $h$ by using
the  rectangular Morley element and the extended high order rectangular Morley element. From the
computational results   listed in tables 3-4, we can see that the two elements both ensure 1/2 order convergence as $\epsilon\rightarrow 0$.
This is conformable to the
theoretical result (\ref{uniformestimate}).

\begin{center}
\vspace{0.3cm} Table 1: The rectangular Morley element \vspace{0.1cm}
\begin{tabular}{|c|c|c|c|c|c|}\hline\
\backslashbox{$\epsilon$}{$h$}& $2^{-2}$& $2^{-3}$& $2^{-4}$ & $2^{-5}$&rate \\
\hline
\hline $2^{0}$    & 0.3899    & 0.1944   & 0.0972   & 0.0486 & 1.00  \\
\hline $2^{-2}$   & 0.3629    & 0.1741   & 0.0862   & 0.0430 & 1.03  \\
\hline $2^{-4}$   & 0.3166    & 0.1020   & 0.0431   & 0.0206 & 1.31  \\
\hline $2^{-6}$   & 0.4165    & 0.1197   & 0.0240   & 0.0070 & 1.96  \\
\hline $2^{-8}$   & 0.4442    & 0.2055   & 0.0544   & 0.0084 & 1.91  \\
\hline $2^{-10}$  & 0.4463    & 0.2243   & 0.1024   & 0.0265 & 1.36  \\
\hline Poisson    & 0.4464    & 0.2258   & 0.1132   & 0.0567 & 0.99  \\
\hline Biharmonic & 0.3923    & 0.1961   & 0.0981   & 0.0491 & 1.00  \\
\hline
\end{tabular}

\vspace{0.3cm}

Table 2: The extended high order rectangular Morley element

\vspace{0.1cm}
\begin{tabular}{|c|c|c|c|c|c|}\hline\
\backslashbox{$\epsilon$}{$h$}& $2^{-2}$& $2^{-3}$& $2^{-4}$ &$2^{-5}$  &rate\\
\hline
\hline $2^{0}$    & 0.2469    & 0.1233   & 0.0615   & 0.0307  & 1.00 \\
\hline $2^{-2}$   & 0.2209    & 0.1093   & 0.0544   & 0.0271  & 1.01 \\
\hline $2^{-4}$   & 0.1154    & 0.0530   & 0.0258   & 0.0128  & 1.06 \\
\hline $2^{-6}$   & 0.0564    & 0.0187   & 0.0077   & 0.0036  & 1.33 \\
\hline $2^{-8}$   & 0.0488    & 0.0126   & 0.0035   & 0.0012  & 1.79 \\
\hline $2^{-10}$  & 0.0483    & 0.0121   & 0.0031   & 0.0008  & 1.97 \\
\hline Poisson    & 0.0482    & 0.0121   & 0.0031   & 0.0008  & 1.99 \\
\hline Biharmonic & 0.2510    & 0.1253   & 0.0625   & 0.0312  & 1.00 \\
\hline
\end{tabular}
\end{center}

\begin{center}
\vspace{0.3cm} 

Table 3: The rectangular Morley element

\vspace{0.1cm}
\begin{tabular}{|c|c|c|c|c|c|}\hline\
\backslashbox{$\epsilon$}{$h$}& $2^{-2}$& $2^{-3}$& $2^{-4}$ & $2^{-5}$&rate \\
\hline
\hline $2^{0}$    & 0.1052    & 0.0514   & 0.0255   & 0.0127 & 1.02  \\
\hline $2^{-2}$   & 0.0554    & 0.0259   & 0.0127   & 0.0063 & 1.05  \\
\hline $2^{-4}$   & 0.0913    & 0.0344   & 0.0106   & 0.0033 & 1.60  \\
\hline $2^{-6}$   & 0.2353    & 0.1070   & 0.0485   & 0.0182 & 1.23  \\
\hline $2^{-8}$   & 0.3065    & 0.2089   & 0.1184   & 0.0543 & 0.83  \\
\hline $2^{-10}$  & 0.3068    & 0.2162   & 0.1525   & 0.1041 & 0.52  \\
\hline
\end{tabular}

\vspace{0.3cm} 

Table 4: The extended high order rectangular Morley element

\vspace{0.1cm}
\begin{tabular}{|c|c|c|c|c|c|}\hline\
\backslashbox{$\epsilon$}{$h$}& $2^{-2}$& $2^{-3}$& $2^{-4}$ & $2^{-5}$ & rate \\
\hline
\hline $2^{0}$    & 0.0196    & 0.0097   & 0.0048   & 0.0024  & 1.01\\
\hline $2^{-2}$   & 0.0734    & 0.0366   & 0.0182   & 0.0091  & 1.01\\
\hline $2^{-4}$   & 0.1554    & 0.0921   & 0.0488   & 0.0247  & 0.88\\
\hline $2^{-6}$   & 0.2352    & 0.1286   & 0.0822   & 0.0496  & 0.75\\
\hline $2^{-8}$   & 0.2785    & 0.1907   & 0.1150   & 0.0641  & 0.71\\
\hline $2^{-10}$  & 0.2772    & 0.1917   & 0.1347   & 0.0937  & 0.52\\
\hline
\end{tabular}


\end{center}


\end{document}